%% file: ECSG_rev_v3.tex
\documentclass[11pt]{article}
\usepackage{amsmath,amsthm,amssymb,amscd}
\usepackage{graphicx}
\usepackage{graphics}
\usepackage{verbatim}

\usepackage{mathrsfs}
\usepackage{color}
\usepackage[top=1in, bottom=1in, left=1.25in, right=1.25in]{geometry}
\usepackage{enumerate}
\RequirePackage{doi}

\newtheorem{theorem}{Theorem}[section]

\newtheorem{observation}[theorem]{Observation}

\newtheorem{lemma}[theorem]{Lemma}

\newenvironment{claim}[1]{\par\noindent{Claim}\space#1}{}
\newenvironment{claimproof}[1]{\par\noindent{Proof:}\space#1}{$\blacksquare$}

\usepackage{thmtools}
\usepackage{thm-restate}

\usepackage{titlesec}

\title{Bounds for the chromatic index of signed multigraphs}

\author{Eckhard Steffen and Isaak H.~Wolf
	 \thanks{Funded by the Deutsche Forschungsgemeinschaft (DFG, German Research Foundation) – 445863039 \newline
Paderborn University,  Department of Mathematics, Warburger Strasse 100, 33098 Paderborn, Germany. Email:~es@upb.de and isaak.wolf@upb.de}}

\begin{document}
\date{}
\maketitle

\begin{abstract}
The paper studies edge-coloring of signed multigraphs and extends classical
Theorems of Shannon \cite{shannon1949theorem} and 
König \cite{konig1916graphen} to signed multigraphs.
	
We prove that the chromatic index of a signed multigraph $(G,\sigma_G)$ is at most 
$\lfloor \frac{3}{2} \Delta(G) \rfloor$. Furthermore,
the chromatic index of a 
balanced signed multigraph $(H,\sigma_H)$ is at most $\Delta(H)  + 1$
and the balanced signed multigraphs with chromatic index
$\Delta(H)$ are characterized. 
\end{abstract}

\section{Introduction}

Edge-coloring of graphs is a classical topic in graph theory and early
results for upper bound of the chromatic index of graphs are
proved by K\"onig and Shannon. K\"onig \cite{konig1916graphen} proved
that the chromatic index of a bipartite graph is equal to
its maximum vertex degree $\Delta$, and Shannon \cite{shannon1949theorem}
proved that it is bounded by $\lfloor \frac{3}{2} \Delta \rfloor$ for arbitrary 
graphs.
 
There are several notions of (vertex-) coloring of signed graphs,
which had been discussed intensively, see \cite{SV_survey} for a survey. 
Recently, a very natural notion of edge-coloring signed graphs had been introduced in 
\cite{behr2020edge, zhang2020edge}. This notion generalizes edge-coloring of graphs. Indeed, for all-negative signed graphs it coincides with edge-coloring
(unsigned) graphs. Motivated by the aforementioned 
results of K\"onig and Shannon, we prove corresponding bounds for signed graphs. 
We show that the chromatic index of a balanced signed graph is at most $\Delta + 1$,
and it is at most $\Delta$, if the graph has a matching, whose removal results in
a graph with even maximum degree (Theorem \ref{thm: Koenig}). We further prove that
Shannon's bound also applies to signed graphs (Theorem \ref{thm: Shannon}).

\section{Definitions and basic results}

For simplicity, we use the term graph instead of multigraph in this paper.
Thus, graphs may contain parallel edges and loops. The vertex set of a graph $G$ 
is denoted by $V(G)$ and its edge set by $E(G)$.
Let $X \subseteq V(G)$ or $X \subseteq E(G)$, the graph induced
by $X$ is denoted by $G[X]$. 

Let $e$ be an edge that is incident to vertices $v$ and $w$. If there is no harm of confusion, then we will also use the term $vw$ to denote $e$.   
Edge $vw$ consists of two half-edges one of which is incident to $v$ and the
other to $w$. The half-edges are denoted by $h_e(v)$ and $h_e(w)$. If $v=w$, then $e$ is a loop. The set of half-edges of $G$ is denoted by $H(G)$
and $H_G(v)$ denotes the set of half-edges that are incident to $v$.
The cardinality of $H_G(v)$ is the degree of $v$ and is denoted by $d_G(v)$. 
Furthermore, $\Delta(G) = \max\{d_G(v) \colon v \in V(G)\}$ is the maximum degree of $G$. The set of vertices of maximum degree in $G$ is denoted by $V_{\Delta(G)}$.
A circuit is a connected 2-regular graph. 

For the definition of an edge-coloring of signed graphs
we will use a natural correspondence between bidirectional graphs and signed graphs. 
Let $G$ be a graph and $\tau_G: H(G) \rightarrow  \{ \pm 1 \}$ be a function. 
This function defines a bidirection of the edges of $G$. If $e=vw$ is an edge, then 
half-edge $h_e(v)$ is oriented away from $v$ if $\tau_G(h_e(v))=1$, and
towards $v$ if $\tau_G(h_e(v))=-1$. Edge $e$ is introverted if $\tau_G(h_e(v)) = \tau_G(h_e(w)) = 1$ and extroverted if $\tau_G(h_e(v)) = \tau_G(h_e(w)) = -1$.

Let $G$ be a graph with bidirection $\tau$. The signed graph $(G,\sigma)$ 
is the graph $G$ together with a function $\sigma : E(G) \rightarrow \{ \pm 1 \}$ with $\sigma(e) = -(\tau(h_e(v))\tau(h_e(w)))$ for each edge $e=vw$.
The function $\sigma$ is called a signature of $G$.
An edge $e$ is negative if $\sigma(e) = -1$, otherwise it is positive.   
The set of negative edges of $(G,\sigma)$ is denoted by $N_{\sigma}$. The set 
$E(G) - N_{\sigma}$ is the set of positive edges. 
By definition, $N_{\sigma}$ consists of those edges which are introverted or extroverted in the bidirected graph.

A circuit is positive if the product of the signs
of its edges is positive, otherwise it is negative. 
A subgraph $(H,\sigma|_{E(H)})$ of $(G,\sigma)$ 
is balanced if all circuits in $(H,\sigma|_{E(H)})$ are positive, otherwise it is unbalanced.
If $\sigma(e)=1$ for all $e \in E(G)$, then $\sigma$ is the all positive signature and it is denoted by $\texttt{\bf 1}$,
and if $\sigma(e)=-1$ for all $e \in E(G)$, then $\sigma$ is the all negative signature and it is denoted by $\texttt{\bf -1}$.

A resigning of a signed graph $(G,\sigma)$ at a vertex $v$ defines a signed graph $(G,\sigma')$ with 
$\tau_G'(h) = -\tau_G(h)$ if $h \in H_G(v)$, $\tau_G'(h) = \tau_G(h)$ otherwise, and 
$\sigma'(e) = -(\tau_G'(h_e(u))\tau_G'(h_e(w)))$ for each edge $e=uw$. Note that resigning at $v$
changes the sign of each edge incident to $v$ but the sign of a loop incident to $v$ is unchanged. Thus, resigning does not change the parity of the number of negative edges in an eulerian subgraph of a signed graph.
Resigning defines an equivalence relation on the set of signatures 
of a graph $G$. We 
say that $(G,\sigma_1)$ and $(G,\sigma_2)$ are equivalent if they can be obtained from each other by a sequence 
of resignings. We also say that $\sigma_1$ and $\sigma_2$ are equivalent signatures of $G$.
In \cite{SignedGraphs} it is proved that two signed graphs are equivalent if and only if they have the same set of negative circuits.

We frequently will use the following well known fact. Let $(G,\sigma)$ be a signed graph and
$S \subseteq E(G)$. If $G[S]$ is a forest, then there is a signature $\sigma'$ that
is equivalent to $\sigma$ and $S \subseteq N_{\sigma'}$. 

It suffices to consider edge-cuts for resigning since  
resigning at every vertex of a set $X \subseteq V(G)$ changes the sign of every edge 
of the edge-cut $\partial_G(X)$ and the signs
of all other edges are unchanged. Thus, Harary's characterization of balanced signed graphs \cite{Balance} has the following consequence.

\begin{observation} \label{Obs: Characterization Balance}
A signed graph $(G,\sigma)$ is balanced if and only if it is equivalent to $(G,\texttt{\bf 1})$. 
\end{observation}

We define a signed graph $(G,\sigma)$ to be antibalanced if it is equivalent to $(G,\texttt{\bf -1})$.
Clearly, $(G,\sigma)$ is antibalanced if and only if the sign product of every even circuit is 1 and it is -1 for every odd circuit.

\subsection{Coloring signed graphs} \label{Subsec: coloring}

Basically we follow the approach of \cite{behr2020edge, zhang2020edge} for 
edge-coloring signed graphs. For simplicity we use the word coloring
instead of edge-coloring since we only consider edge-coloring in this paper. 
For technical reasons we will use elements
of symmetric sets for coloring as it is introduced in \cite{cappello2021symmetric}. 
A set $S$ together with a sign ``$-$'' is a symmetric set 
if it satisfies the following conditions:	
\begin{enumerate}
	\item $s \in S$ if and only if $-s \in S$.
	\item If $s=s'$, then $-s=-s'$.
	\item $s = -(-s)$.
\end{enumerate}

An element $s$ of a symmetric set $S$ is self-inverse if $s=-s$. 
A symmetric set with self-inverse elements $0_1, \dots, 0_t$
and non-self-inverse elements $\pm s_1, \dots, \pm s_k$ is
denoted by $S_{2k}^t$. Clearly, $|S_{2k}^t|=t+2k$. 

Let $(G,\sigma)$ be a signed graph with orientation $\tau$ and 
$t,k$ be two non-negative integers not both equal to $0$. 
A function $c \colon E(G) \rightarrow S^t_{2k}$ is an
$S^t_{2k}$-coloring of $(G,\sigma)$ (with orientation $\tau$)
if
$$ \tau(h_{e_1}(v))c(e_1) \not = \tau(h_{e_2}(v))c(e_2)$$
for every vertex $v \in V(G)$ and any two different edges 
$e_1, e_2 \in E_G(v)$. An example is given in Figure~\ref{fig:K4_coloring}.

\begin{figure}[htbp]
\centering
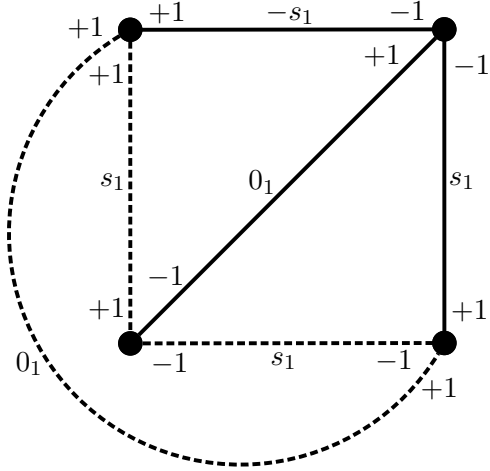
\caption{An example of an $S^1_2$-coloring. Negative edges are drawn as dashed lines.}
\label{fig:K4_coloring}
\end{figure}


By definition, if $(G,\sigma)$ and $(G,\sigma')$ are equivalent, then 
$(G,\sigma)$ admits an $S^t_{2k}$-coloring if and only if
$(G,\sigma')$ admits an $S^t_{2k}$-coloring.
We will often (implicitly) use the following statements.

\begin{observation} \label{obs:color half-edges}
	Let $(G,\sigma)$ be a signed graph with orientation $\tau$. 
	If $(G,\sigma)$ has an $S^t_{2k}$-coloring $c$, then 
	$(G,\sigma)$ has no negative loop, and for each edge 
	$e=vw \colon$ $\tau(h_{e}(v))c(e) = \tau(h_{e}(w)) c(e)$ if and only if $e$ is negative or $c(e)$ is self-inverse.
\end{observation}

Let $c$ be an $S^t_{2k}$-coloring of $(G,\sigma)$ and $v \in V(G)$. 
We say that color $a$ appears at $v$, if $\tau(h_{e}(v))c(e) = a$
for an edge $e=vw$. Otherwise, we say that $v$ misses $a$. 
Indeed, $c$ induces an $S^t_{2k}$-coloring $c_H$ of
the half-edges of $(G,\sigma)$, where $c_H(h_e(v)) = \tau(h_{e}(v))c(e)$.

For $S^0_{2k}$ and $S^1_{2k}$ our definition coincides with the definition of
an edge-coloring of simple signed graphs given in \cite{behr2020edge, zhang2020edge}.
We will need a slight extension of some basic results which are 
observed by Behr \cite{behr2020edge}.

\begin{lemma} \label{lem: basics}
	 1. A signed path has an $S^0_2$-coloring. \cite{behr2020edge} \\
    2. A signed circuit has an $S^0_2$-coloring if and only if it is positive.  \cite{behr2020edge} \\
	3. A negative signed circuit has an $S^1_{2}$-coloring
		that colors precisely one negative edge with color $0$. 
\end{lemma}

\begin{proof} We prove item 3. A negative circuit $(C,\sigma)$ has a negative edge $e$. 
	Then $(C-e,\sigma|_{E(C-e)})$ is a path and
	therefore, it has an $S^0_{2}$-coloring by statement 1. This
	coloring can be extended to a $S^1_{2}$-coloring of $(C,\sigma)$ by coloring $e$ with color $0$. 
\end{proof}

Let $c$ be an $S^t_{2k}$-coloring of $(G,\sigma)$. 
A color class of a self-inverse color is called a self-inverse color class
and a color class of a non-self-inverse color is called non-self-inverse.
By Observation \ref{obs:color half-edges}, a self-inverse color class 
is a matching in $G$ and a non-self-inverse color class $c^{-1}(\pm s)$ 
consist of paths and balanced circuits, by Lemma \ref{lem: basics}.
This is due to the fact that the two sets $c^{-1}(s)$ and $c^{-1}(-s)$ 
cannot not separated from each other in view of the half-edges. 
Note that the induced $S^t_{2k}$coloring $c_H$
of the half-edges may use color $-s$ even for edges of $c^{-1}(s)$. 
This fact allows a simple variant of Kempe-switching in signed graphs. 
Let $D$
be a component of a color class. Let $c'(e) = -c(e)$ if $e \in D$
and $c'(e)= c(e)$ otherwise. Then $c'$ is an $S^t_{2k}$-coloring of $(G,\sigma)$. We say that $c'$ is obtained from $c$ by resigning a color at $D$.

We will also follow the approach of \cite{behr2020edge, zhang2020edge}
for defining the signed chromatic index of a signed graph. 
Coloring an edge with a self-inverse color annuls the sign of the edge. Thus,
it is appropriate to minimize the number of self-inverse elements in the definition
of the chromatic index of a signed graph. Since the number of self-inverse 
elements determines the parity of a symmetric set we need two sets for coloring. 
For $i \in \{0,1\}$ let $\chi'_i(G,\sigma)$ be the minimum $2k + i$ such that 
$(G,\sigma)$ admits an $S^i_{2k}$-coloring. The chromatic index
of $(G,\sigma)$ is $\min\{\chi'_0(G,\sigma), \chi'_1(G,\sigma)\}$ and
it is denoted by $\chi'(G,\sigma)$.  
Note that by 
Observation \ref{obs:color half-edges}, $\chi'(G,\texttt{\bf{-1}}) = \chi'(G)$.

Let $L$ be a spanning subgraph of a graph $G$ with $\Delta(L) \leq 2$.
The edge set of $L$ is called a layer of $G$. The following
result is folklore and can easily be deduced from Petersen's 2-factorization theorem
for even regular graphs \cite{petersen1891theorie}. 

\begin{lemma} \label{lem: layers} 
The edge set of any graph $G$ can be decomposed into $\lceil \frac{\Delta(G)}{2} \rceil$ layers.
\end{lemma}

We also say that $G$ is decomposable into $\lceil \frac{\Delta(G)}{2} \rceil$ layers.

\section{Upper bounds for the chromatic index}

The chromatic index of an antibalanced signed 
graph is equal to the chromatic index of its underlying unsigned 
graph.
Therefore, the case $\sigma = \texttt{\bf{-1}}$ of Theorem \ref{thm: Shannon}
is Shannon's Theorem \cite{shannon1949theorem}.
Theorem \ref{thm: Koenig} extends  
K\"onig's Theorem \cite{konig1916graphen} that the chromatic index
of a bipartite graph is equal to its maximum degree to balanced signed graphs.
 
Our proofs are based on a decomposition of the edge set of $(G,\sigma)$ into signed layers. 
Shannon \cite{shannon1949theorem} used this idea, which he attributed to Foster, 
for the proof of his theorem for graphs with even maximum degree.  
Our proof then also gives a proof by this method for Shannon's theorem 
for graphs with odd maximum degree.

\begin{theorem} \label{thm: Koenig}
	If $(G,\sigma)$ is a balanced signed graph, 
	then $\chi'(G,\sigma) \leq \Delta(G)+1$. In particular,
	$\chi'(G,\sigma) = \Delta(G)$ if and only if 
	$G$ has a (possible empty)  
	matching $M$ such that $\Delta(G-M)$ is even.
	The bound is best possible. 
\end{theorem}

\begin{proof} 
	By Lemma \ref{lem: layers}, $(G,\sigma)$ 
	decomposes into $k = \lceil \frac{\Delta(G)}{2} \rceil$ layers. 
	Each layer is balanced and therefore, 
$S^0_2$-colorable by Lemma \ref{lem: basics}. The
$S^0_2$-colorings of the $k$ layers easily combine to an $S^0_{2k}$-coloring of 
$(G,\sigma)$ and the first statement is proved.

If $G$ has a matching $M$ such that $\Delta(G-M)$ is even, say $\Delta(G-M)=2t$, 
then it follows as above that $(G-M, \sigma|_{E(G-M)})$ has 
an $S^0_{2t}$-coloring $c$. If $M \not = \emptyset$, then $c$ can be extended to an $S^1_{2t}$-coloring of $(G,\sigma)$ by coloring the elements of $M$ with color $0$.

If $\chi'(G,\sigma) = \Delta(G)$, then we consider first the case when
$\Delta(G)=2k+1$. Thus, 
$(G,\sigma)$ has an $S^1_{2k}$-coloring $c$ and therefore, $c^{-1}(0)$ is a matching 
in $G$ that covers every vertex of maximum degree. Thus,
$\Delta(G-c^{-1}(0)) = \Delta(G)-1$ is even and the statement is proved. The case $\Delta(G)$ even is trivial.

Let $r > 0$ be an odd integer. Then $\chi'(G,\texttt{\bf 1}) = r+1$, 
for every $r$-regular graph $G$, which does not have a 1-factor. 
\end{proof}

Next we will prove a Shannon-type theorem for signed graphs. 

\begin{theorem} \label{thm: Shannon}
	For each signed graph $(G,\sigma): \chi'(G,\sigma) \leq \lfloor \frac{3}{2} \Delta(G) \rfloor$.
\end{theorem}

\begin{proof} Let $(G,\sigma)$ be a signed graph. 
	By Lemma \ref{lem: basics}, a layer has an $S^1_2$-coloring that colors precisely one negative edge in each negative circuit (if there is any) with the self-inverse color and all other edges with the non-self-inverse color. We assume that a layer $L_i$ has such a coloring $c_i$ which uses colors $0_i$ and $\pm s_i$.
	
	\begin{claim}{1:} 
	If $\Delta(G) = 2t$, then $(G,\sigma)$ has an $S^t_{2t}$-coloring $c$
	with $\bigcup_{i=1}^t c^{-1}(0_i) \subseteq N_{\sigma}$.
	\end{claim}

\begin{claimproof} 
By Lemma \ref{lem: layers}, $E(G)$ decomposes into $t$ layers $L_1, \dots, L_t$.
By our assumptions,
layer $L_i$ has an $S^1_{2}$-coloring $c_i$ which uses the colors $0_i$ and $\pm s_i$
and each edge colored with $0_i$ is negative.  Now, 
the colorings $c_1, \dots, c_t$ 
easily combine to an $S^{t}_{2t}$-coloring $c$ of $(G,\sigma)$ with 
$\bigcup_{i=1}^t c^{-1}(0_i) \subseteq N_{\sigma}$. 
\end{claimproof}

\begin{claim}{2:} If $\Delta(G)=2t+1$ and $V_{\Delta}$ is an independent set, then there is an equivalent signature $\sigma '$  such that $(G,\sigma')$ has an $S_{2t}^t$-coloring $c$ with $\bigcup_{i=1}^t c^{-1}(0_i) \subseteq N_{\sigma'}$.
\end{claim}

\begin{claimproof}
	By Hall's Theorem, there is a matching $M$ of $m$ edges $e_1, \ldots, e_m$ that covers all vertices of $V_{\Delta}=\{x_1, \ldots x_m\}$. Let $e_j=x_jy_j$, then $d_G(y_j)\leq 2t$. Since $M$ is a matching, there is a signature $\sigma'$ equivalent to $\sigma$ such that $M \subseteq N_{\sigma'}$. By Lemma~\ref{lem: layers}, $G-M$ decomposes into $t$ layers $L_1, \ldots, L_t$. For each $j \in \{1, \ldots ,m\}$, add $e_j$ to a layer $L \in \{L_1, \ldots, L_t\}$ such that $y_j$ is incident to at most one half-edge of $L$, which is possible since $d_G(y_j) \leq 2t$. We obtain a decomposition of $G$ into $t$ sub-cubic graphs $L_1', \ldots L_t'$. We claim that for each of these subgraphs there is an $S_{2}^1$-coloring in which every edge colored with the self-inverse color belongs to $N_{\sigma'}$.
	
	For $i \in \{1,\ldots,t\}$ let $L_i'=L_i+M_i$, i.e. $M_i \subseteq M$ is the set of edges added to $L_i$.

By our assumption, there is an $S_2^1$-coloring $c_i$ of $(L_i,\sigma' |_{E(L_i)})$ that colors precisely one negative edge with color $0_i$ in each negative circuit and all other edges with $\pm s_i$. Note that, if $x_jy_j \in M_i$, then $y_j$ misses color $0_i$ and at least one of $\pm s_i$. Furthermore, since $V_{\Delta}$ is an independent set, every edge colored with $0_i$ is adjacent to at most one edge of $M_i$. Thus, by resigning colors $s_i$ and $-s_i$ at some unbalanced circuits of $L_i$ if necessary, we can transform $c_i$ to an $S_2^1$-coloring $c_i'$ of $(L_i,\sigma' |_{E(L_i)})$ such that
	\begin{itemize}
		\item[1.] $(c_i')^{-1}(0_i) \subseteq N_{\sigma'}$,
		\item[2.] if $x_jy_j \in M_i$ is incident with an edge colored with $0_i$, then $x_j$ and $y_j$ either both miss color $s_i$ or both miss color $-s_i$.
	\end{itemize}
	Now, $c_i'$ can be extended to an $S_2^1$-coloring of $(L_i',\sigma' |_{E(L_i')})$ by coloring every edge of $M_i$ adjacent to a $0_i$-colored edge with either $s_i$ or $-s_i$ and all other edges of $M_i$ with $0_i$.
	
	Thus, for every $i \in \{1, \ldots, t\}$ there is an $S_{2}^1$-coloring of $L_i'$ in which every edge colored with the self-inverse color is negative. These colorings easily combine to an $S_{2t}^t$-coloring $c$ of $(G,\sigma')$ with $\bigcup_{i=1}^t c^{-1}(0_i) \subseteq N_{\sigma'}$.
\end{claimproof}

\begin{claim}{3:} 
		If $\Delta(G) = 2t+1$ and $V_{\Delta}$ is not an independent set, then there is an equivalent signature $\sigma'$ such that $(G,\sigma')$ has an $S^{t+1}_{2t}$-coloring $c$ with $\bigcup_{i=1}^{t+1} c^{-1}(0_i) \subseteq N_{\sigma'}$.
\end{claim}
\begin{claimproof}
Case 1: $G$ has a matching $M$ that covers all vertices of degree $2t+1$. Since $M$ is a matching, there is an equivalent signature $\sigma'$ with $M \subseteq N_{\sigma'}$. By Claim 1, $(G-M,\sigma'|_{E(G)-M})$ has an $S^t_{2t}$-coloring $c$ with $\bigcup_{i=1}^t c^{-1}(0_i) \subseteq N_{\sigma'}$. Coloring the edges of $M$ with a new self-inverse color $0_{t+1}$ gives the desired coloring.

Case 2: $G$ does not have a matching that covers all vertices of degree $2t+1$. 
Let $M_1$ be a matching of $G$ that covers the maximum number of vertices of degree $2t+1$ in $G$. Thus, the vertices of degree $2t+1$ form an independent set in $H_1=G-M_1$. Let $M_2$ be a matching in $H_1$ that covers all vertices of degree $2t+1$ in $H_1$. The components of $G[M_1 \cup M_2]$ are isomorphic to $K_2$ or $K_{1,2}$. Thus, there is an equivalent signature $\sigma'$ with $M_1 \cup M_2 \subseteq N_{\sigma'}$. Since all edges of $M_2$ are negative, there is an $S^{t}_{2t}$-coloring $c$ of $(H_1,\sigma'|_{E(H_1)})$ with $\bigcup_{i=1}^t c^{-1}(0_i) \subseteq N_{\sigma'}$ by the same argumentation as in the proof of Claim 2. Coloring the edges of $M_1$ with a new self-inverse color $0_{t+1}$ gives the desired coloring.  
\end{claimproof}	

Claims 1-3 show that in any case,  there is an equivalent signature $\sigma'$ such that $(G,\sigma')$ has an
$S^t_{2t}$-coloring or an $S^{t+1}_{2t}$-coloring (only if $\Delta = 2t+1$) 
where every edge
that is colored by a self-inverse color is negative. 

Note that every self-inverse color class is a matching and therefore,
the union of two self-inverse color classes consists of paths and even circuits. 
Since all these edges are negative, the circuits are balanced. Thus, the union of two 
self-inverse color classes can be colored with two (new) non-self-inverse colors 
$\pm a$. 

By recoloring pairs of self-inverse color classes 
of the above colorings we obtain the appropriate 
colorings of $(G,\sigma')$. If $t$ is odd, then an $S^t_{2t}$-coloring
transforms to an $S^1_{3t-1}$-coloring and an $S^{t+1}_{2t}$-coloring to an 
$S^0_{3t+1}$-coloring. 
If $t$ is even, then an $S^t_{2t}$-coloring
transforms to an $S^0_{3t}$-coloring and an $S^{t+1}_{2t}$ to an 
$S^1_{3t}$-coloring. 

Thus, if $\Delta(G) = 2t$, then 
$\chi'(G,\sigma') \leq 3t \leq \lfloor \frac{3}{2} \Delta(G) \rfloor$, and
if $\Delta(G) = 2t+1$, then 
$\chi'(G,\sigma') \leq 3t+1 \leq \lfloor \frac{3}{2} \Delta(G) \rfloor$. Since $\sigma$ and $\sigma'$ are equivalent, Theorem \ref{thm: Shannon}
is proved. 
\end{proof}	 

The bound of Theorem \ref{thm: Shannon} is sharp. Since 
$\chi'(G,\texttt{\bf -1}) = \chi'(G)$, the bound is attained by 
the fat triangles, where any two vertices are connected by the same number of 
parallel edges. 


\end{document}

%% file: K_4_coloring.pdf_tex
\begingroup%
  \makeatletter%
  \providecommand\color[2][]{%
    \errmessage{(Inkscape) Color is used for the text in Inkscape, but the package 'color.sty' is not loaded}%
    \renewcommand\color[2][]{}%
  }%
  \providecommand\transparent[1]{%
    \errmessage{(Inkscape) Transparency is used (non-zero) for the text in Inkscape, but the package 'transparent.sty' is not loaded}%
    \renewcommand\transparent[1]{}%
  }%
  \providecommand\rotatebox[2]{#2}%
  \newcommand*\fsize{\dimexpr\f@size pt\relax}%
  \newcommand*\lineheight[1]{\fontsize{\fsize}{#1\fsize}\selectfont}%
  \ifx\svgwidth\undefined%
    \setlength{\unitlength}{180.88466041bp}%
    \ifx\svgscale\undefined%
      \relax%
    \else%
      \setlength{\unitlength}{\unitlength * \real{\svgscale}}%
    \fi%
  \else%
    \setlength{\unitlength}{\svgwidth}%
  \fi%
  \global\let\svgwidth\undefined%
  \global\let\svgscale\undefined%
  \makeatother%
  \begin{picture}(1,0.95845591)%
    \lineheight{1}%
    \setlength\tabcolsep{0pt}%
    \put(0,0){\includegraphics[width=\unitlength,page=1]{K_4_coloring.pdf}}%
    \put(0.16920467,0.79890325){\makebox(0,0)[lt]{\lineheight{1.25}\smash{\begin{tabular}[t]{l}$+1$\end{tabular}}}}%
    \put(0.30035449,0.19760367){\makebox(0,0)[lt]{\lineheight{1.25}\smash{\begin{tabular}[t]{l}$-1$\end{tabular}}}}%
    \put(0.86228588,0.1452597){\makebox(0,0)[lt]{\lineheight{1.25}\smash{\begin{tabular}[t]{l}$+1$\end{tabular}}}}%
    \put(0.9256691,0.3068166){\makebox(0,0)[lt]{\lineheight{1.25}\smash{\begin{tabular}[t]{l}$+1$\end{tabular}}}}%
    \put(0.74305474,0.83842968){\makebox(0,0)[lt]{\lineheight{1.25}\smash{\begin{tabular}[t]{l}$+1$\end{tabular}}}}%
    \put(0.29486125,0.92933186){\makebox(0,0)[lt]{\lineheight{1.25}\smash{\begin{tabular}[t]{l}$+1$\end{tabular}}}}%
    \put(0.12603274,0.8972879){\makebox(0,0)[lt]{\lineheight{1.25}\smash{\begin{tabular}[t]{l}$+1$\end{tabular}}}}%
    \put(0.79524056,0.93024735){\makebox(0,0)[lt]{\lineheight{1.25}\smash{\begin{tabular}[t]{l}$-1$\end{tabular}}}}%
    \put(0.92952617,0.81920334){\makebox(0,0)[lt]{\lineheight{1.25}\smash{\begin{tabular}[t]{l}$-1$\end{tabular}}}}%
    \put(0.16925809,0.31139429){\makebox(0,0)[lt]{\lineheight{1.25}\smash{\begin{tabular}[t]{l}$+1$\end{tabular}}}}%
    \put(0.29209869,0.37909164){\makebox(0,0)[lt]{\lineheight{1.25}\smash{\begin{tabular}[t]{l}$-1$\end{tabular}}}}%
    \put(0.76791646,0.20401246){\makebox(0,0)[lt]{\lineheight{1.25}\smash{\begin{tabular}[t]{l}$-1$\end{tabular}}}}%
    \put(0.53924207,0.93207847){\makebox(0,0)[lt]{\lineheight{1.25}\smash{\begin{tabular}[t]{l}$-s_1$\end{tabular}}}}%
    \put(0.92054328,0.58772961){\makebox(0,0)[lt]{\lineheight{1.25}\smash{\begin{tabular}[t]{l}$s_1$\end{tabular}}}}%
    \put(0.54918533,0.21241046){\makebox(0,0)[lt]{\lineheight{1.25}\smash{\begin{tabular}[t]{l}$s_1$\end{tabular}}}}%
    \put(0.19309254,0.58493022){\makebox(0,0)[lt]{\lineheight{1.25}\smash{\begin{tabular}[t]{l}$s_1$\end{tabular}}}}%
    \put(0.50340825,0.57836335){\makebox(0,0)[lt]{\lineheight{1.25}\smash{\begin{tabular}[t]{l}$0_1$\end{tabular}}}}%
    \put(0.01822419,0.19959288){\makebox(0,0)[lt]{\lineheight{1.25}\smash{\begin{tabular}[t]{l}$0_1$\end{tabular}}}}%
  \end{picture}%
\endgroup%